%% file: version1.tex
\documentclass[12pt]{article}
\usepackage{amsmath,amssymb,amsthm,latexsym,amsbsy,amsfonts,mathrsfs}
\usepackage{color,verbatim}
\usepackage{pstricks}
\usepackage{tabularx}
\usepackage{rotating}
\usepackage{graphicx,epsfig}

\input{ams_notation}

\input{env}
\title{Weak solutions of the Landau--Lifshitz--Bloch equation
 \thanks{This work was supported by the 
 Australian Research 
 Council grant DP140101193.}}
\author{Kim Ngan Le
        \thanks{School of Mathematics and Statistics,
        The University of New South Wales,
         Sydney 2052, Australia
         Email:
                 {\tt n.le-kim@unsw.edu.au}
                }     
       }
\newtheorem{notation}{Notation}[section]

\newcommand{\iprod}[1]{\langle#1\rangle}
\newcommand{\bigiprod}[1]{\bigl\langle#1\bigr\rangle}

\begin{document}

\maketitle
\pagenumbering{arabic}
\begin{abstract}
The Landau--Lifshitz--Bloch (LLB) equation is a formulation of dynamic 
micromagnetics valid at all temperatures, treating both the transverse and 
longitudinal relaxation components important for high-temperature 
applications. 
We study LLB equation in case the temperature 
raised higher than the Curie temperature. The existence of weak solution 
is showed and its regularity properties are also discussed. In this way, 
we lay foundations for the rigorous theory of LLB equation that is currently 
not available.

{\bf Key words}: Landau--Lifshitz--Bloch, quasilinear parabolic equation, ferromagnetism

{\bf AMS suject classifications}: 82D40, 35K59, 35R15
\end{abstract}

\section{Introduction}
Micromagnetic modeling has proved itself as a widely used tool, 
complimentary in many respects to experimental measurements. 
The Landau--Lifshitz--Gilbert (LLG) equation~\cite{LL35,Gil55} provides a basis for this modeling, 
especially where the dynamical behaviour is concerned. 
According to this theory, at temperatures below the critical (so-called Curie) temperature, 
the magnetization $\vecm(t,\vecx)\in\mS^2$, where $\mS^2$ is the unit sphere 
in $\R^3$,
for $t>0$ and $\vecx\in D\subset \R^d$, $d=1,2,3$, satisfies the following LLG equation
\begin{equation}\label{eq: LLG}
\frac{\partial \vecm}{\partial t} 
= 
\lambda_1 \vecm\times \vecH_{\text{eff}}
-
\lambda_2\vecm\times(\vecm\times \vecH_{\text{eff}}),
\end{equation}
where $\times$ is the vector cross product in $\R^3$ and $\vecH_{\text{eff}}$ 
is the so-called effective field.

However, for high temperatures the model must be replaced by a more thermodynamically consistent 
approach such as the Landau--Lifshitz--Bloch (LLB) equation~\cite{Garanin1991,Garanin97}. 
The LLB equation essentially interpolates between 
the LLG equation at low temperatures and 
the Ginzburg-Landau theory of phase transitions. It is valid not only 
below but also above the Curie temperature $T_{\text{c}}$. 
An important property of the LLB equation is that the magnetization magnitude 
is no longer conserved but is a dynamical variable~\cite{Garanin97,Evans12}. 
The spin polarization $\vecu(t,\vecx)\in\R^3$, ($\vecu=\vecm/m_s^0$, $\vecm$ is magnetization and 
$m_s^0$ is the saturation magnetization value at $T=0$), 
for $t>0$ and $\vecx\in D\subset \R^d$, $d=1,2,3$, satisfies the following LLB equation
\begin{equation}\label{eq: LLB}
\frac{\partial \vecu}{\partial t} 
= 
\gamma \vecu\times \vecH_{\text{eff}} 
+
L_1\frac{1}{|\vecu|^2}(\vecu\cdot\vecH_{\text{eff}})\vecu
-
L_2\frac{1}{|\vecu|^2}\vecu\times(\vecu\times\vecH_{\text{eff}}).
\end{equation}
Here, $|\cdot|$ is the Euclidean norm in $\R^3$, $\gamma>0$ is the gyromagnetic ratio, and $L_1$ 
and $L_2$ are the longitudial  and transverse  
damping parameters, respectively.

LLB micromagnetics has become a real alternative to LLG micromagnetics 
for temperatures which are close to the Curie temperature 
($T\gtrsim\tfrac34T_{\text{c}}$). 
This is realistic for some novel exciting phenomena, such as 
light-induced demagnetization with powerfull femtosecond (fs) lasers~\cite{Atxitia2007}. 
During this process the electronic temperature is normally 
raised higher than $T_{\text{c}}$. 
Micromagnetics based on the LLG equation cannot work under these 
circumstances while micromagnetics based on the LLB equation has proved 
to describe correctly the observed fs magnetization dynamics.

In this paper, we consider a deterministic form of a ferromagnetic LLB equation, 
in which the temperature 
$T$ is raised higher than $T_{\text{c}}$, and  as a consequence the longitudial $L_1$ and transverse $L_2$ damping parameters 
are equal. The effective field $\vecH_{\text{eff}}$ is given by
\[
 \vecH_{\text{eff}} = \Delta\vecu - 
 \frac{1}{\chi_{||}}
 \bigg(1+\frac{3}{5}\frac{T}{T-T_c}|\vecu|^2\bigg)\vecu, 
\]
where $\chi_{||}$ is the longitudinal susceptibility.

By using the vector triple product identity 
$\veca\times(\vecb\times\vecc) = \vecb(\veca\cdot\vecc)-\vecc(\veca\cdot\vecb)$, we get
\[
\vecu\times(\vecu\times\vecH_{\text{eff}})
=
(\vecu\cdot\vecH_{\text{eff}})\vecu
-
|\vecu|^2\vecH_{\text{eff}},
\]
and from property $L_1=L_2=:\kappa_1$, 
we can rewrite~\eqref{eq: LLB} as follows
\begin{equation}\label{eq: LLB2}
\frac{\partial \vecu}{\partial t} 
= 
\kappa_1\Delta\vecu
+
\gamma\vecu\times  \Delta\vecu
-
\kappa_2(1+\mu|\vecu|^2)\vecu,
\quad \text{with }
\kappa_2 := \frac{\kappa_1}{\chi_{||}},
\quad
\mu := \frac{3T}{5(T-T_c)}.
\end{equation}
So the LLB equation we are going to study in this paper is equation~\eqref{eq: LLB2} with   real positive 
coefficients $\kappa_1,\kappa_2,\gamma,\mu$, initial data 
$\vecu(0,\vecx)=\vecu_0(\vecx)$ and subject to 
homogeneous Neumann boundary conditions.

Various results on existence of global weak
solutions  of the LLG equation~\eqref{eq: LLG} are proved in~\cite{CarbouFabrie01,AloSoy92}. More complete
lists can be found in~\cite{Cimrak_survey,GuoDing08,KruzikProhl06}.
Furthermore, there is also some research about the weak solution of its stochastic version (i.e., the effective field 
is perturbed by a Gaussian noise), such as in~\cite{BrzGolJer12,bookBanBrzNekPro13}. 
It should be mentioned that the proof of existence in~\cite{Banas2013,BanBrzPro13,GoldysLeTran2016} is a
constructive proof, namely an approximate solution can be
computed.

To the best of our knowledge the analysis of the LLB equation is an open problem at present. 
In this paper, we introduce a definition of  weak solutions of the LLB equation. 
By introducing the Faedo--Galerkin approximations and using the method of compactness, we 
prove the existence of weak solutions for the LLB equation.

This paper is organized as follows. 
In Section~\ref{sec: nota} we introduce the notations and formulate the main result 
(Theorem~\ref{theo: main}) on the existence of the weak solution of~\eqref{eq: LLB2} 
as well as some regularity properties. In Section~\ref{sec: FG} we introduce the Faedo--Galerkin 
approximations and prove for them some uniform bounds in various norms. 
In Section~\ref{sec: Exist}, we use the method of compactness to show the existence of 
a weak solution and prove the main theorem. Finally, in the Appendix we collect, 
for the reader's convenience, some facts scattered in the literature that are used 
in the course of the proof.

\section{Notation and the formulation of the main result}\label{sec: nota}
Before presenting the definition of a weak solution to the LLB equation~\eqref{eq: LLB2}, 
it is necessary to introduce some function spaces.

The function spaces $\mH^1(D,\R^3)=:\mH^1$  are defined as follows:
\begin{align*}
\mH^1(D,\R^3)
&=
\left\{ \vecu\in\mL^2(D,\R^3) : \frac{\p\vecu}{\p
x_i}\in \mL^2(D,\R^3)\quad\text{for } i=1,2,3.
\right\}.
\end{align*}
Here, $\mL^p(D,\R^3)=:\mL^p$ with $p>0$ is the usual space of $p^\text{th}$-power 
Lebesgue
integrable functions defined on $D$ and taking values in $\R^3$.
Throughout this paper, we denote a scalar product in a Hilbert space $H$ by
$\inpro{\cdot}{\cdot}_H$ and its associated norm by $\|\cdot\|_H$. The dual brackets between 
a space $X$ and its dual $X^*$ will be denoted $_{X}\!\iprod{\cdot,\cdot}_{X^*}$. 

\begin{definition}\label{def: weakso}
Given $T>0$, a weak solution 
 $\vecu : [0,T] \goto \mH^1\cap \mL^4$ to~\eqref{eq: LLB2} satisfies
\begin{align}\label{eq: weakLLB}
\iprod{\vecu(t),\vecphi}_{\mL^2} 
=
&\iprod{\vecu_0,\vecphi}_{\mL^2}
-\kappa_1\int_0^t \iprod{\nabla\vecu(s),\nabla\vecphi}_{\mL^2}\ds 
-\gamma\int_0^t \iprod{\vecu(s)\times\nabla\vecu(s),\nabla\vecphi}_{\mL^2}\ds\nn \\
&-\kappa_2\int_0^t \iprod{(1+\mu|\vecu|^2(s))\vecu(s),\vecphi}_{\mL^2}\ds,
\end{align}
for every $\vecphi\in \C_0^{\infty}(D)$ and $t\in[0,T]$.
\end{definition}
Now we  can formulate the main result of this paper.
\begin{theorem}\label{theo: main}
Let $D\subset\R^d$ be an open bounded domain with $C^m$ extension property and assume that $d<2m$.  
For $T>0$ and for the initial data $\vecu_0\in\mH^1$, there exists a weak solution of~\eqref{eq: LLB2} 
such that 
\begin{enumerate}
\item[(a)] for every $t\in[0,T]$,
\begin{align}\label{eq: weakLLB2}
\vecu(t)
=
&\vecu_0
+\kappa_1\int_0^t \Delta\vecu(s)\ds
+\gamma\int_0^t \vecu(s)\times\Delta\vecu(s)\ds\nn\\
&-\kappa_2\int_0^t (1+\mu|\vecu|^2(s))\vecu(s)\ds \quad\text{in $\mL^{3/2}$,}
\end{align}

\item[(b)]
for every $\alpha\in(0,\tfrac14]$, $\vecu\in C^{\alpha}([0,T],\mL^{3/2})$,
\item[(c)] $\sup_{t\in[0,T]}\|\vecu(t,\cdot)\|_{\mL^2}<\infty$.
\end{enumerate}
\end{theorem}
\begin{remark}
The notation $\Delta\vecu$ and $\vecu\times\Delta\vecu$ will 
be defined in the Notations~\ref{no: nota1}--\ref{no: nota2}.
\end{remark}

\section{Faedo-Galerkin Approximation}\label{sec: FG}
Let $A=-\Delta$ be the negative Laplace operator. 
From~\cite[Theorem 1, p. 335]{Evans1998}, there exists an orthonormal basis 
$\{\vece_i\}_{i=1}^{\infty}$ of $\mL^2$, consisting 
of eigenvectors for operator $A$, such that $\vece_i\in\C^{m}(D)\cap\mL^{\infty}$ 
for all $i=1,2,$\dots and
\[
-\Delta\vece_i = \lambda_i\vece_i,\quad 
\vece_i = 0 \text{ on }\partial D,
\]
where $\lambda_i>0$ for $i=1,2,$\dots are eigenvalues of $A$.
Let $S_n:=\text{span}\{\vece_1,\cdots,\vece_n\}$ and 
$\Pi_n$ be the orthogonal projection from $\mL^2$ onto $S_n$, 
defined by: for $\vecv\in\mL^2$
\begin{equation}\label{eq: Pi_n}
\iprod{\Pi_n\vecv,\vecphi}_{\mL^2} 
= \iprod{\vecv,\vecphi}_{\mL^2},\quad\forall \vecphi\in S_n.
\end{equation}
By taking $\vecphi = \Pi_n\vecv$ in the above equation, we obtain an upper bound for 
the projection operator $\Pi_n$ in $\mL^2$,
\begin{equation}\label{eq: boundPi_n}
\|\Pi_n\vecv\|_{\mL^2}
\leq
\|\vecv\|_{\mL^2}
\quad\forall\vecv\in S_n.
\end{equation}
We note that  $\Pi_n$ is a self-adjoint operator on $\mL^2$, indeed, 
from ~\eqref{eq: Pi_n}, for $\vecv,\vecw\in\mL^2$ 
there holds
\[
\iprod{\vecw,\Pi_n\vecv}_{\mL^2}
=
\iprod{\Pi_n\vecv,\Pi_n\vecw}_{\mL^2}
=
\iprod{\vecv,\Pi_n\vecw}_{\mL^2}.
\]

We are now looking for approximate solution 
$\vecu_n(\cdot,t)\in S_n:=\text{span}\{\vece_1,\cdots,\vece_n\}$ of equation~\eqref{eq: LLB2} satisfying
\begin{equation}\label{eq: GaLLB}
\frac{\partial \vecu_n}{\partial t}
-\kappa_1\Delta\vecu_n
-\gamma\Pi_n\bigl(\vecu_n\times\Delta\vecu_n\bigr)
+\kappa_2\Pi_n\bigl((1+\mu|\vecu_n|^2)\vecu_n\bigr)
=0,
\end{equation}
with $\vecu_n(\cdot,0) = \vecu_{0n}$, where $\vecu_{0n}\in S_n$ is an approximation 
of $\vecu_0$. 
Since equation~\eqref{eq: GaLLB} is equivalent to an ordinary differential equation in $\R^n$, 
the existence of a local solution to~\eqref{eq: GaLLB} 
is a consequence of the following lemma.
\begin{lemma}
For $n\in\N$, define the maps:
\begin{align*}
&F^1_n: S_n\ni \vecv \mapsto \Delta\vecv\in S_n,\\
&F^2_n: S_n\ni\vecv\mapsto \Pi_n(\vecv\times\Delta\vecv)\in S_n,\\
&F^3_n: S_n\ni\vecv\mapsto \Pi_n((1+\mu|\vecv|^2)\vecv)\in S_n.
\end{align*}
Then $F^1_n$ is globally Lipschitz and $F^2_n$, $F^3_n$ are locally Lipschitz.
\end{lemma}
\begin{proof}
For any $\vecv\in S_n$ we have 
\begin{equation*}
\vecv = \sum_{i=1}^n \inpro{\vecv}{\vece_i}_{\mL^2}\vece_i
\quad\text{and}\quad
-\Delta\vecv = \sum_{i=1}^n \lambda_i\inpro{\vecv}{\vece_i}_{\mL^2}\vece_i.
\end{equation*}
By using the triangle inequality, the orthonormal property 
of $\{\vece_i\}_{i=1}^n$ and H\"older's inequality, for any $\vecu,\vecv\in S_n$ we obtain
\begin{align*}
\|F^1_n(\vecu)-F^1_n(\vecv)\|_{\mL^2}
&=
\|\Delta\vecu-\Delta\vecv\|_{\mL^2}
=
\|\sum_{i=1}^n \lambda_i\inpro{\vecu-\vecv}{\vece_i}_{\mL^2}\vece_i\|_{\mL^2}\\
&\leq
\sum_{i=1}^n \lambda_i\bigl|\inpro{\vecu-\vecv}{\vece_i}_{\mL^2}\bigr|
\leq
\bigl(\sum_{i=1}^n \lambda_i\bigr)\|\vecu-\vecv\|_{\mL^2},
\end{align*}
then the globally Lipschitz property of $F^1_n$ follows immediately.

From~\eqref{eq: boundPi_n} and the triangle inequality, there holds
\begin{align*}
\|F^2_n(\vecu)-F^2_n(\vecv)\|_{\mL^2}
&=
\|\Pi_n(\vecu\times\Delta\vecu-\vecv\times\Delta\vecv)\|_{\mL^2}
\leq
\|\vecu\times\Delta\vecu-\vecv\times\Delta\vecv\|_{\mL^2}\\
&\leq
\|\vecu\times(\Delta\vecu-\Delta\vecv)\|_{\mL^2}
+
\|(\vecu-\vecv)\times\Delta\vecv\|_{\mL^2}\\
&\leq
\|\vecu\|_{\mL^{\infty}}
\|F^1_n(\vecu)-F^1_n(\vecv)\|_{\mL^2}
+
\|(\vecu-\vecv)\|_{\mL^2}
\|\Delta\vecv\|_{\mL^{\infty}}.
\end{align*}
Since $F^1_n$ is globally Lipschitz and 
the fact that all norms are equivalent in the finite dimensional space $S_n$, $F^2_n$ is locally Lifshitz.

Similarly, the local Lipschitz property of $F^3_n$ follows from the estimate,
\begin{align*}
\|F^3_n(\vecu)-F^3_n(\vecv)\|_{\mL^2}
&\leq
\|\vecu-\vecv\|_{\mL^2}
+
\mu\|\Pi_n(|\vecu|^2\vecu-|\vecv|^2\vecv)\|_{\mL^2}\\
&\leq
\|\vecu-\vecv\|_{\mL^2}
+
\mu\||\vecu|^2\vecu-|\vecv|^2\vecv\|_{\mL^2}\\
&\leq
\|\vecu-\vecv\|_{\mL^2}
+
\mu\||\vecu|^2(\vecu-\vecv)\|_{\mL^2}
+
\mu\|(\vecu-\vecv)\cdot(\vecu+\vecv)\,\vecv\|_{\mL^2}\\
&\leq
\bigl(
1
+
\mu\||\vecu|^2\|_{\mL^{\infty}}
+
\mu\|\vecu+\vecv\|_{\mL^{\infty}}\|\vecv\|_{\mL^{\infty}}\bigr)
\|\vecu-\vecv\|_{\mL^2},
\end{align*}
which complete the proof of this lemma.
\end{proof}
We now proceed to priori estimates on the approximate solution $\vecu_n$.
\begin{lemma}\label{lem: appSo_sta}
For each $n=1,2,$\dots and every $t\in[0,T]$,
\begin{equation*}
\|\vecu_n(t)\|_{\mL^2}^2 
+
2\kappa_1\int_0^T \|\nabla\vecu_n(t)\|_{\mL^2}^2 \dt
+
2\kappa_2
\int_0^T \bigl(\|\vecu_n(t)\|_{\mL^2}^2 + \mu\|\vecu_n(t)\|_{\mL^4}^4\bigr)\dt
\leq
\|\vecu_n(0)\|_{\mL^2}^2, 
\end{equation*}
and
\[
\|\nabla\vecu_n(t)\|_{\mL^2}^2
+
2\kappa_1\int_0^T \|\Delta\vecu_n(t)\|_{\mL^2}^2 \dt
\leq
\|\nabla\vecu_n(0)\|_{\mL^2}^2.
\]
\end{lemma}
\begin{proof}
Taking the inner product of both sides of ~\eqref{eq: GaLLB} 
with $\vecu_n(t)\in S_n$, integrating by parts with respect to $\vecx$, and using 
$(\veca\times\vecb)\cdot\vecb=0$ and the fact that $\Pi_n$ is self-adjoint, 
we obtain
\begin{align*}
\frac{1}{2}\frac{\partial}{\partial t}\|\vecu_n(t)\|_{\mL^2}^2 
+
\kappa_1\|\nabla\vecu_n(t)\|_{\mL^2}^2
+
\kappa_2
\bigiprod{(1+\mu|\vecu_n|^2)\vecu_n,\vecu_n(t)}_{\mL^2} = 0.
\end{align*}
The first result follows by integrating both sides of the above equation 
with respect to $t$.

In a similar fashion, we next take the inner product of both sides of ~\eqref{eq: GaLLB} 
with $\Delta\vecu_n(t)\in S_n$, and then integrate by parts with respect to $\vecx$
to arrive at
\begin{align*}
\frac{1}{2}\frac{\partial}{\partial t}\|\nabla\vecu_n(t)\|_{\mL^2}^2 
+
\kappa_1\|\Delta\vecu_n(t)\|_{\mL^2}^2
&+
\kappa_2
\bigiprod{(1+\mu|\vecu_n|^2)\nabla\vecu_n,\nabla\vecu_n(t)}_{\mL^2}\\
&+
\kappa_2
\bigiprod{2\mu(\vecu_n\cdot\nabla\vecu_n)\vecu_n,\nabla\vecu_n(t)}_{\mL^2}
=0
\end{align*}
Integrating both sides 
with respect to $t$, 
we obtain
\begin{align*}
\|\nabla\vecu_n(t)\|_{\mL^2}^2 
+
2\kappa_1
\int_0^t
\|\Delta\vecu_n(s)\|_{\mL^2}^2\ds
&+
2\kappa_2
\int_0^t
\int_D
(1+\mu|\vecu_n|^2)(\nabla\vecu_n)^2\,d\vecx\ds\\
&+
2\kappa_2\mu
\int_0^t
\int_D
(\vecu_n\cdot\nabla\vecu_n)^2\,d\vecx\ds
=\|\nabla\vecu_n(0)\|_{\mL^2}^2, 
\end{align*}
and the second result follows immediately.
\end{proof}

The following upper bounds for $\vecu_n\times\Delta\vecu_n$ and $(1+\mu|\vecu_n|^2)\vecu_n$ 
are a consequence of Lemma~\ref{lem: appSo_sta}.
\begin{lemma}\label{lem: bound_un1}
There exists  a  constant $C$, which does not depend on 
$n=1,2,$\dots, such that
\begin{align*}
\int_0^T\|\vecu_n(t)\times\Delta\vecu_n(t)\|^2_{\mL^{3/2}}\dt 
\leq C\quad\text{and}\quad
\int_0^T\|(1+\mu|\vecu_n|^2(t))\vecu_n(t)\|^2_{\mL^2}\dt
\leq C.
\end{align*}
\end{lemma}
\begin{proof}
By H\"older's inequality and the Sobolev imbedding of $\mH^1$ into $\mL^6$~\cite{Friedman1969}
we have
\begin{align*}
\|\vecu_n(t)\times\Delta\vecu_n(t)\|_{\mL^{3/2}}
\leq
\|\vecu_n(t)\|_{\mL^6}
\|\Delta\vecu_n(t)\|_{\mL^2}
\leq
C
\|\vecu_n(t)\|_{\mH^1}
\|\Delta\vecu_n(t)\|_{\mL^2}.
\end{align*}
We use Lemma~\ref{lem: appSo_sta} to obtain the first result,
\begin{align*}
\int_0^T \|\vecu_n(t)\times\Delta\vecu_n(t)\|^2_{\mL^{3/2}}\dt
\leq
C\sup_{t\in[0,T]}\|\vecu_n(t)\|_{\mH^1}^2
\int_0^T
\|\Delta\vecu_n(t)\|_{\mL^2}^2\dt
\leq C.
\end{align*}
Similarly, from Lemma~\ref{lem: appSo_sta} and the Sobolev imbedding of $\mH^1$ into $\mL^6$, we have
\begin{equation}\label{eq: un3}
\|\vecu_n^3(t)\|^2_{\mL^2}
=
\|\vecu_n(t)\|^6_{\mL^6}
\leq
\|\vecu_n(t)\|^6_{\mH^1}
\leq
C,
\end{equation}
so
\begin{align*}
\|\bigl(1+\mu|\vecu_n|^2(t)\bigr)\vecu_n(t)\|^2_{\mL^2}
\leq
2\|\vecu_n(t)\|^2_{\mL^2}
+
2\mu^2\|\vecu_n^3(t)\|^2_{\mL^2}
\leq
C,
\end{align*}
and the second result follows immediately.
\end{proof}

Equation~\eqref{eq: GaLLB} can be written in the following way 
as an approximation of equation~\eqref{eq: LLB2},
\begin{align}\label{eq: GaLLB2}
\vecu_n(t)
&=\vecu_n(0)
+\kappa_1\int_0^t\Delta\vecu_n\ds
+\gamma\int_0^t\Pi_n\bigl(\vecu_n\times\Delta\vecu_n\bigr)\ds
-\kappa_2\int_0^t\Pi_n\bigl((1+\mu|\vecu_n|^2)\vecu_n\bigr)\ds\\
&=\vecu_n(0) + \kappa_1\vecB_{n,1}(t)
+\gamma\vecB_{n,2}(t)
+\kappa_2\vecB_{n,3}(t)
.\nn
\end{align}

Before proving the uniform bound of $\{\vecu_n\}$, we define the following 
fractional power space~\cite[Definiton 1.4.7]{Henry1981}.
\begin{definition}\label{def: fracspace}
Put $A_1:=I+A$. For any real number $\beta>0$, we define 
the Hilbert space 
\[
X^{\beta} := \bigl\{\vecphi\in\mL^2: \|A_1^{\beta}\vecphi\|_{\mL^2}<\infty\bigr\},
\]
where 
$A_1^{\beta}\vecphi := 
\sum_{i=1}^{\infty}
(1+\lambda_i)^{\beta}
\inpro{\vecphi}{\vece_i}_{\mL^2}\vece_i,$ 
with the graph norm $\|\cdot\|_{X^{\beta}}=\|A_1^{\beta}\cdot\|_{\mL^2}$. 
The dual space of $X^{\beta}$ is denoted by $X^{-\beta}$.
\end{definition}
The following lemma states an upper bound for 
the projection operator $\Pi_n$ in $X^{-\beta}$.
\begin{lemma}\label{lem: boundPi_n}
For any $\beta>0$ and $\vecv\in\mL^2$ there holds
\[
\|\Pi_n\vecv\|_{X^{-\beta}}\leq \|\vecv\|_{X^{-\beta}}.
\]
\end{lemma}
\begin{proof}
The proof of this lemma can be found in~\cite{ZdzisLiang2014};  
for the reader's convenience we recall the proof as follows. 

For $\vecv\in\mL^2$, by using~\eqref{eq: Pi_n} we obtain
\begin{align}\label{eq: bound1}
\|\Pi_n\vecv\|_{X^{-\beta}}
&=
\sup_{\|\vecw\|_{X^{\beta}}\leq 1}
\left|_{X^{-\beta}}\!\iprod{\Pi_n\vecv,\vecw}_{X^{\beta}}\right|
=
\sup_{\|\vecw\|_{X^{\beta}}\leq 1}
\left|\iprod{\Pi_n\vecv,\vecw}_{\mL^2}\right|\nn\\
&=
\sup_{\|\vecw\|_{X^{\beta}}\leq 1}
\left|\iprod{\vecv,\Pi_n\vecw}_{\mL^2}\right|.
\end{align}
Since 
\[\|\Pi_n\vecw\|^2_{X^{\beta}} = \sum_{i=1}^{n}
(1+\lambda_i)^{2\beta}
\inpro{\vecw}{\vece_i}_{\mL^2}^2
\leq
\sum_{i=1}^{\infty}
(1+\lambda_i)^{2\beta}
\inpro{\vecw}{\vece_i}_{\mL^2}^2
=
\|\vecw\|^2_{X^{\beta}},
\]
the set $\{\vecw\in X^{\beta}: \|\vecw\|_{X^{\beta}}\leq 1 \}$ 
is a subset of the set 
$\{\vecw\in X^{\beta}: \|\Pi_n\vecw\|_{X^{\beta}}\leq 1 \}$. Hence, 
from~\eqref{eq: bound1} there holds
\[
\|\Pi_n\vecv\|_{X^{-\beta}}
\leq
\sup_{\|\Pi_n\vecw\|_{X^{\beta}}\leq 1}
\left|\iprod{\vecv,\Pi_n\vecw}_{\mL^2}\right|
\leq
\|\vecv\|_{X^{-\beta}},
\]
which completes the proof of the lemma.
\end{proof}

We now prove a uniform bound for $\{\vecu_n\}$ in $H^1(0,T;X^{-\beta})$.
\begin{lemma}\label{lem: bound_un2}
Let $D\subset\R^d$ be an open bounded domain with the $C^m$ extension property.  
Given $\beta>\frac{d}{6m}$, there exists  a  
constant $C$, which does not depend on 
$n$ such that
\begin{align}
&\|\vecB_{n,2}\|_{H^1(0,T;X^{-\beta})} \leq C,\label{eq: bound4}\\
&\|\vecB_{n,3}\|_{H^1(0,T;\mL^2)} \leq C,\label{eq: bound5}
\end{align}
and 
\begin{equation}\label{eq: bound6}
\|\vecu_n\|_{H^1(0,T;X^{-\beta})} \leq C,
\end{equation}
with $\vecB_{n,2}$ and $\vecB_{n,3}$ are defined in~\eqref{eq: GaLLB2}.
\end{lemma}
\begin{proof}
Since $\beta>\frac{d}{6m}$, by using Lemma~\ref{lem: Ap1} we infer that
$X^{\beta}$ is continuously embedded in $\mL^3$. 
Thus we have the  continuous imbedding 
\begin{equation}\label{eq: embed1}
\mL^{3/2}\hookrightarrow X^{-\beta}.  
\end{equation}

\underline{Proof of~\eqref{eq: bound4}:} 
By using Lemma~\ref{lem: boundPi_n},~\eqref{eq: embed1} and the first result of Lemma~\ref{lem: bound_un1} 
we deduce
\begin{align}\label{eq: bound2}
\int_0^T\|\frac{\partial}{\partial t}\vecB_{n,2}(t)\|^2_{X^{-\beta}}\dt
&=
\int_0^T\|\Pi_n\bigl(\vecu_n(t)\times\Delta\vecu_n(t)\bigr)\|^2_{X^{-\beta}}\dt\nn\\
&\leq
C
\int_0^T\|\vecu_n(t)\times\Delta\vecu_n(t)\|^2_{X^{-\beta}}\dt\nn\\
&\leq
C
\int_0^T\|\vecu_n(t)\times\Delta\vecu_n(t)\|^2_{\mL^{3/2}}\dt
\leq
C.
\end{align}
In the same maner, we estimate $\vecB_{n,2}$ in the norm of $L^2(0,T;X^{-\beta})$ 
as follows. 
Since $\vece_i\in\mL^{\infty}$ for $i=1,\cdots,n$, we see from Lemma~\ref{lem: appSo_sta} that
\[
\int_0^t\int_D\bigg|\bigl(\vecu_n(s)\times\Delta\vecu_n(s)\bigr)\cdot\vece_i\bigg|\dvx\ds
\leq
\|\vece_i\|_{\mL^{\infty}}
\|\vecu_n\|_{L^2(0,T;\mL^2)}
\|\Delta\vecu_n\|_{L^2(0,T;\mL^2)}
< \infty,
\]
and thus from Fubini's theorem there holds
\begin{equation}\label{eq: Fubi1}
\int_0^t \Pi_n\bigl(\vecu_n(s)\times\Delta\vecu_n(s)\bigr)\ds
=
 \Pi_n\bigg(\int_0^t\vecu_n(s)\times\Delta\vecu_n(s)\ds\bigg).
\end{equation}
By using~\eqref{eq: Fubi1}, Lemma~\ref{lem: boundPi_n},~\eqref{eq: embed1} 
and Minkowski's inequality, 
we deduce
\begin{align*}
\int_0^T\|\vecB_{n,2}(t)\|^2_{X^{-\beta}}\dt
&=
\int_0^T\bigg\|\int_0^t\Pi_n\bigl(\vecu_n(s)\times\Delta\vecu_n(s)\bigr)\ds\bigg\|^2_{X^{-\beta}}\dt\\
&=
\int_0^T\bigg\|\Pi_n\bigg(\int_0^t\vecu_n(s)\times\Delta\vecu_n(s)\ds\bigg)\bigg\|^2_{X^{-\beta}}\dt\\
&\leq
\int_0^T\bigg\|\int_0^t\vecu_n(s)\times\Delta\vecu_n(s)\ds\bigg\|^2_{X^{-\beta}}\dt\\
&\leq
\int_0^T\bigg\|\int_0^t\vecu_n(s)\times\Delta\vecu_n(s)\ds\bigg\|^2_{\mL^{3/2}}\dt\\
&\leq
\int_0^T\bigg(\int_0^t\|\vecu_n(s)\times\Delta\vecu_n(s)\|_{\mL^{3/2}}\ds\bigg)^2\dt.
\end{align*}
Thus, it follows from H\"older's inequality and  the first result of Lemma~\ref{lem: bound_un1} that
\begin{align}\label{eq: bound3}
\int_0^T\|\vecB_{n,2}(t)\|^2_{X^{-\beta}}\dt
\leq
\int_0^T t\int_0^t\|\vecu_n(s)\times\Delta\vecu_n(s)\|^2_{\mL^{3/2}}\ds\dt
\leq
\int_0^T tC\dt = CT^2.
\end{align}
The first result~\eqref{eq: bound4} follows immediately 
from~\eqref{eq: bound2} and~\eqref{eq: bound3}.

\underline{Proof of~\eqref{eq: bound5}:} Using the same technique as in the proof of~\eqref{eq: bound4}, 
we prove~\eqref{eq: bound5} as follows.

From~\eqref{eq: boundPi_n} and the second result of Lemma~\ref{lem: bound_un1}, we deduce
\begin{align}\label{eq: bound7}
\int_0^T\|\frac{\partial}{\partial t}\vecB_{n,3}(t)\|^2_{\mL^2}\dt
&=
\int_0^T\|\Pi_n\bigl((1+\mu|\vecu_n|^2(t))\vecu_n(t)\bigr)\|^2_{\mL^2}\dt\nn\\
&\leq
\int_0^T\|(1+\mu|\vecu_n|^2(t))\vecu_n(t)\|^2_{\mL^2}\dt
\leq
C.
\end{align}
Since $\vece_i\in\mL^2$ for $i=1,\cdots,n$, we see from Lemma~\ref{lem: bound_un1} that
\[
\int_0^t\int_D\bigg|\bigl(1+\mu|\vecu_n(s)|^2\bigr)\vecu_n(s)\cdot\vece_i\bigg|\dvx\ds
\leq
t^{1/2}\|\vece_i\|_{\mL^2}
\|\bigl(1+\mu|\vecu_n|^2\bigr)\vecu_n\|_{L^2(0,T;\mL^2)}
< \infty,
\]
and thus from Fubini's theorem there holds
\begin{equation}\label{eq: Fubi2}
\int_0^t \Pi_n\bigg(\bigl(1+\mu|\vecu_n(s)|^2\bigr)\vecu_n(s)\bigg)\ds
=
 \Pi_n\bigg(\int_0^t\bigl(1+\mu|\vecu_n(s)|^2\bigr)\vecu_n(s)\ds\bigg).
\end{equation}
By using~\eqref{eq: Fubi2} and~\eqref{eq: boundPi_n},the  Minkowski and  H\"older inequalities, and
the second result of Lemma~\ref{lem: bound_un1},
we infer that
\begin{align}\label{eq: bound8}
\int_0^T\|\vecB_{n,3}(t)\|^2_{\mL^2}\dt
&=
\int_0^T\|\int_0^t\Pi_n\bigl((1+\mu|\vecu_n|^2(s))\vecu_n(s)\bigr)\ds\|^2_{\mL^2}\dt\nn\\
&=
\int_0^T\|\Pi_n\bigl(\int_0^t(1+\mu|\vecu_n|^2(s))\vecu_n(s)\ds\bigr)\|^2_{\mL^2}\dt\nn\\
&\leq
\int_0^T\|\int_0^t(1+\mu|\vecu_n|^2(s))\vecu_n(s)\ds\|^2_{\mL^2}\dt\nn\\
&\leq
\int_0^T\bigl(\int_0^t\|(1+\mu|\vecu_n|^2(s))\vecu_n(s)\|_{\mL^2}\ds\bigr)^2\dt\nn\\
&\leq
\int_0^T t\int_0^t\|(1+\mu|\vecu_n|^2(s))\vecu_n(s)\|^2_{\mL^2}\ds\dt\nn\\
&\leq
\int_0^T tC\dt = CT^2.
\end{align}
Thus,~\eqref{eq: bound5} follows from~\eqref{eq: bound7} and~\eqref{eq: bound8}.

\underline{Proof of~\eqref{eq: bound6}:} 
From Lemma~\ref{lem: appSo_sta}, $\Delta\vecu_n$ is uniformly bounded in 
$L^2\bigl(0,T;\mL^2\bigr)$. By using the same arguments as in the proof of~\eqref{eq: bound5}, 
we also deduce 
\begin{equation}\label{eq: bound9}
\|\vecB_{n,1}\|_{H^1(0,T;\mL^2)}
\leq C.
\end{equation}
Since $\mL^2\hookrightarrow \mL^{3/2}$ we see from~\eqref{eq: embed1} that  
$\mL^2\hookrightarrow X^{-\beta}$ and thus 
$H^1\bigl(0,T;\mL^2\bigr)\hookrightarrow H^1\bigl(0,T;X^{-\beta}\bigr)$. 
It follows from~\eqref{eq: bound9} and~\eqref{eq: bound5} that $\vecB_{n,1}$ 
and $\vecB_{n,3}$ are uniformly bounded in $H^1\bigl(0,T;X^{-\beta}\bigr)$.
Together with~\eqref{eq: bound4} we have 
\[
\|\vecu_n\|_{H^1(0,T;X^{-\beta})} \leq C,
\]
which complete the proof of this lemma.
\end{proof}
\section{Existence of a weak solution}\label{sec: Exist}
In this section, by using the method of compactness, 
we show that there is a subsequence of $\{\vecu_n\}$ whose limit 
is a weak solution of~\eqref{eq: LLB2}.

Firstly, in the following lemma we prove the existence of 
a convergent subsequence of 
$\vecu_n$ in a functional space.
\begin{lemma}\label{lem: conversub}
Let $D\subset\R^d$ be an open bounded domain with the $C^m$ extension property 
and let $\vecu_n$ be
the solution of~\eqref{eq: GaLLB}
for $n=1,2,$\dots. Assume that $d<2m$, then 
there exist a subsequence of $\{\vecu_n\}$ 
(still denoted by $\{\vecu_n\}$) and 
$\vecu\in C([0,T];X^{-\bar\beta })\cap L^{\bar p}(0,T;\mL^4)$ such that
\begin{equation}\label{eq: converu}
\vecu_n \goto \vecu \text { strongly in }
\C([0,T];X^{-\bar\beta })\cap L^{\bar p}(0,T;\mL^4),
\end{equation}
where $\bar\beta >\frac{d}{6m}$ and $\bar p\geq 4$. Furthermore, 
\begin{equation}\label{eq: converu2}
\vecu_n \goto \vecu \text { weakly in }
L^2(0,T;\mH^1).
\end{equation}
\end{lemma}
\begin{proof}
From~\eqref{eq: bound6}, the sequence $\{\vecu_n\}_n$ is uniformly bounded in 
$H^1(0,T;X^{-\beta})$ with given $\beta>\frac{d}{6m}$. 
For each 
$p\in[2,\infty)$, thanks to Lemma~\ref{lem: Ap2} we have the continuous 
imbeddings
\[
H^1(0,T;X^{-\beta})\hookrightarrow \mW^{\alpha,p}(0,T;X^{-\beta})
\quad\text{ if }\alpha\in(0,\tfrac12)\text{ and } \frac{1}{2}>\alpha-\frac{1}{p},
\]
so by Lemma~\ref{lem: appSo_sta} the sequence
$\{\vecu_n\}_n$ is uniformly bounded in 
$ W^{\alpha,p}(0,T;X^{-\beta})\cap L^p(0,T;\mH^1)$.

From~\cite[Theorem 1.4.8]{Henry1981} , $X^{\nu}$ is compactly embedded in 
$X^{\nu'}$ whenever $\nu$ and $\nu'$ are real numbers with $\nu>\nu'$.
Since $\mH^1 = X^{1/2}$, there exists $\gamma\in[-\beta,\tfrac12)$ such that 
the embeddings 
\[
 \mH^1\hookrightarrow X^\gamma \hookrightarrow X^{-\beta}
 \,\text{ are compact.}
\]
By using Lemmas~\ref{lem: Ap3}--\ref{lem: Ap4}, we deduce the compact embeddings
\begin{align}
W^{\alpha,p}(0,T;X^{-\beta})\cap L^p(0,T;\mH^1)
&\hookrightarrow 
L^p(0,T;X^{\gamma}),\label{eq: embed2}\\
\mW^{\alpha,p}(0,T;X^{-\beta})
&\hookrightarrow
C([0,T];X^{-\bar\beta})
\,\text{if }\bar\beta>\beta \text{ and } \alpha p>1.\label{eq: embed4}
\end{align}
From Lemma~\ref{lem: Ap1}, $X^{\gamma}$ is continuously embedded in $\mL^q$ 
when $\gamma>\frac{d(q-2)}{2mq}$, so
\begin{equation}\label{eq: embed3}
L^p(0,T;X^{\gamma})
\hookrightarrow
L^p(0,T;\mL^q)\quad
\text{when } \gamma>\frac{d(q-2)}{2mq}.
\end{equation}
It follows from~\eqref{eq: embed2},~\eqref{eq: embed4} and~\eqref{eq: embed3} that
if 
\begin{equation}\label{eq: cond}
\bar\beta>\beta>\frac{d}{6m},
\quad
\frac{1}{2}>\alpha-\frac{1}{p}>0
\quad
\text{and}\quad \frac{d(q-2)}{2mq}<\frac{1}{2},
\end{equation}
then the embedding
\[
W^{\alpha,p}(0,T;X^{-\beta})\cap L^p(0,T;\mH^1)
\hookrightarrow C([0,T];X^{-\bar\beta})\cap L^p(0,T;\mL^q)
\quad\text{is compact.}
\]
In what follows, we choose $p=\bar p \geq 4,q=4,\bar\beta >\frac{d}{6m}$. Thus, with the assumption 
$d<2m$ the condition~\eqref{eq: cond} holds.
It follows that there exist a subsequence of $\{\vecu_n\}$ 
(still denoted by $\{\vecu_n\}$) and 
$\vecu\in C([0,T];X^{-\bar\beta })\cap L^{\bar p}(0,T;\mL^4)$ such that
\begin{equation*}
\vecu_n \goto \vecu \text { strongly in }
\C([0,T];X^{-\bar\beta })\cap L^{\bar p}(0,T;\mL^4).
\end{equation*}
Furthermore, from Lemma~\ref{lem: appSo_sta}, the sequence $\{\vecu_n\}_n$ is uniformly bounded in 
$L^2(0,T;\mH^1)$. Thus, there exists a subsequence of $\{\vecu_n\}$ 
(still denoted by $\{\vecu_n\}$) such that
\begin{equation*}
\vecu_n \goto \vecu \text { weakly in }
L^2(0,T;\mH^1),
\end{equation*}
which completes the proof of this lemma.
\end{proof}
In the remaining part of this paper, we will choose $\bar p = 8$ in Lemma~\ref{lem: conversub}.

Secondly, we find the limits of sequences $\bigl\{\Pi_n(\vecu_n\times\Delta\vecu_n)\bigr\}_n$ and 
 $\bigl\{\Pi_n((1+ |\vecu_n|^2)\vecu_n)\bigr\}_n$ and their relationship with $\vecu$ 
 in the following lemmas. 

Since the Banach spaces $\mL^2(0,T;\mL^{3/2})$ and 
$\mL^2(0,T;X^{-\beta})$ are all reflexive, 
from Lemmas~\ref{lem: bound_un1}--\ref{lem: bound_un2} and 
by the Banach-Alaoglu Theorem there exist subsequences of $\{\vecu_n\times\Delta\vecu_n\}$ 
and of $\{\Pi_n\bigl(\vecu_n\times\Delta\vecu_n\bigr)\}$ (still denoted by 
$\{\vecu_n\times\Delta\vecu_n\}$, $\{\Pi_n\bigl(\vecu_n\times\Delta\vecu_n\bigr)\}$, respectively); and
$Z\in\mL^2(0,T;\mL^{3/2})$, $\bar Z\in \mL^2(0,T;X^{-\beta})$ such that
\begin{align}
\vecu_n\times\Delta\vecu_n
&\goto \vecZ \text{ weakly in } \mL^2(0,T;\mL^{3/2})\label{eq: converZ}\\
\Pi_n\bigl(\vecu_n\times\Delta\vecu_n\bigr)
&\goto \bar \vecZ \text{ weakly in } \mL^2(0,T;X^{-\beta})\label{eq: converZbar}.
\end{align}
\begin{lemma}\label{lem: conver1}
If $\vecZ$ and $\bar \vecZ$ defined as above, then $\vecZ=\bar \vecZ$ in $\mL^2(0,T;X^{-\beta})$.
\end{lemma}
\begin{proof}
From~\eqref{eq: embed1},  
we infer that $\vecZ\in \mL^2(0,T;X^{-\beta})$.
For every $n\in\N$, let us denote $X^{\beta}_n:=\{\Pi_n\vecx:\vecx\in X^{\beta}\}=S_n$ with the norm inherited 
from $X^{\beta}$. Then from Lemma~\ref{lem: Ap5}, 
$\cup_{n=1}^{\infty}X^{\beta}_n$ is dense $X^{\beta}$ and thus 
$\cup_{n=1}^{\infty}\mL^2(0,T;X^{\beta}_n)$ is dense $\mL^2(0,T;X^{\beta})$. Hence, 
it is sufficient to prove that for any $\vecphi_m\in \mL^2(0,T;X^{\beta}_m)$,
\[
_{\mL^2(0,T;X^{-\beta})}\!\iprod{\bar \vecZ,\vecphi_m}_{\mL^2(0,T;X^{\beta})} 
=
_{\mL^2(0,T;X^{-\beta})}\!\iprod{\vecZ,\vecphi_m}_{\mL^2(0,T;X^{\beta})} .
\]
For this aim let us fix $m\in\N$ and $\vecphi_m\in \mL^2(0,T;X^{\beta}_m)$. 
Since $X^{\beta}_m\subset X^{\beta}_n$ for any $n\geq m$, we have
\begin{align*}
_{\mL^2(0,T;X^{-\beta})}\!\iprod{\Pi_n\bigl(\vecu_n\times\Delta\vecu_n\bigr),\vecphi_m}_{\mL^2(0,T;X^{\beta})} 
&=
\int_0^T\,
_{X^{-\beta}}\!\iprod{\Pi_n\bigl(\vecu_n(t)\times\Delta\vecu_n(t)\bigr),\vecphi_m}_{X^{\beta}} \dt\\
&=                                               
\int_0^T\,
\iprod{\Pi_n\bigl(\vecu_n(t)\times\Delta\vecu_n(t)\bigr),\vecphi_m}_{\mL^2}\dt\\
&=
\int_0^T\,
\iprod{\bigl(\vecu_n(t)\times\Delta\vecu_n(t)\bigr),\Pi_n\vecphi_m}_{\mL^2}\dt\\
&=
\int_0^T\,
\iprod{\bigl(\vecu_n(t)\times\Delta\vecu_n(t)\bigr),\vecphi_m}_{\mL^2}\dt\\
&=
_{\mL^2(0,T;X^{-\beta})}\!\iprod{\bigl(\vecu_n\times\Delta\vecu_n\bigr),\vecphi_m}_{\mL^2(0,T;X^{\beta})}.
\end{align*}
Hence the result follows by taking the limit as $n$ tends to infinity of the above equation and 
using~\eqref{eq: converZ}--\eqref{eq: converZbar}.
\end{proof}
\begin{lemma}\label{lem: conver2}
For any $\vecphi\in \mW^{1,4}(D)\cap X^{\beta}$, there holds
\begin{align}
&\lim_{n\goto\infty}
\int_0^T\,
_{X^{-\beta}}\!\iprod{\Pi_n\bigl(\vecu_n(t)\times\Delta\vecu_n(t)\bigr),\vecphi}_{X^{\beta}} \dt
=
-
\int_0^T
\iprod{\vecu(t)\times\nabla\vecu(t),\nabla\vecphi}_{\mL^2}\label{eq: conver3}\\
\text{and }\quad
&\lim_{n\goto\infty}
\int_0^T
\iprod{\Pi_n\bigl((1+\mu|\vecu_n|^2(t))\vecu_n(t)\bigr),\vecphi}_{\mL^2}\dt
=
\int_0^T
\iprod{(1+\mu|\vecu|^2(t))\vecu(t),\vecphi}_{\mL^2}\dt\label{eq: conver4}
\end{align}
\end{lemma}
\begin{proof}

\underline{Proof of~\eqref{eq: conver3}:}
From~\eqref{eq: converZ}--\eqref{eq: converZbar}, Lemma~\ref{lem: conver1}, 
and 
\[
\iprod{\vecu_n(t)\times\Delta\vecu_n(t),\vecphi}_{\mL^2}=-\iprod{\vecu_n(t)\times\nabla\vecu_n(t),\nabla\vecphi}_{\mL^2}, 
\]
it is sufficient to prove that 
\begin{equation}\label{eq: conver5}
\lim_{n\goto\infty} 
\int_0^T
\iprod{\vecu_n(t)\times\nabla\vecu_n(t),\nabla\vecphi}_{\mL^2}\dt
= 
\int_0^T
\iprod{\vecu(t)\times\nabla\vecu(t),\nabla\vecphi}_{\mL^2}\dt.
\end{equation}
By using the triangle and H\"older inequalities together with Lemma~\ref{lem: appSo_sta}, we see that
\begin{align*}
\bigg|
\int_0^T
&\iprod{\vecu_n(t)\times\nabla\vecu_n(t),\nabla\vecphi}_{\mL^2}\dt
-
\int_0^T
\iprod{\vecu(t)\times\nabla\vecu(t),\nabla\vecphi}_{\mL^2}\dt
\bigg|\\
&\leq
\bigg|\int_0^T\iprod{(\vecu_n(t)-\vecu(t))\times\nabla\vecu_n(t),\nabla\vecphi}_{\mL^2}\dt\bigg|\\
&\quad+
\bigg|\int_0^T\iprod{\vecu(t)\times(\nabla\vecu_n(t)-\nabla\vecu(t)),\nabla\vecphi}_{\mL^2}\dt\bigg|\\
&\leq
\|\vecu_n-\vecu\|_{L^4(0,T;\mL^4)} \|\nabla\vecu_n\|_{L^2(0,T;\mL^2)}\|\nabla\vecphi\|_{L^4(0,T;\mL^4)}\\
&\quad
+
\bigg|\int_0^T\iprod{\nabla\vecu_n(t)-\nabla\vecu(t),\nabla\vecphi\times\vecu(t)}_{\mL^2}\dt\bigg|\\
&\leq
C \|\vecu_n-\vecu\|_{L^4(0,T;\mL^4)}
+
\bigg|\int_0^T\iprod{\nabla\vecu_n(t)-\nabla\vecu(t),\nabla\vecphi\times\vecu(t)}_{\mL^2}\dt\bigg|.
\end{align*}
Hence,~\eqref{eq: conver3}  follows by passing to the limit as $n$ tends to infinity 
of the above inequality and using~\eqref{eq: converu}--\eqref{eq: converu2}, 
 noting that $\nabla\vecphi\times\vecu\in L^2(0,T;\mL^2)$ since $\vecu\in L^4(0,T;\mL^4)$.

\underline{Proof of~\eqref{eq: conver4}:} 
Since $\Pi_n$ is a self-adjoint operator on $\mL^2$, we have
\[
\iprod{\Pi_n\bigl((1+\mu|\vecu_n|^2(t))\vecu_n(t)\bigr),\vecphi}_{\mL^2}
=
\iprod{\vecu_n,\vecphi}_{\mL^2}
+
\mu\iprod{|\vecu_n|^2(t)\vecu_n(t),\Pi_n\vecphi}_{\mL^2},
\]
so from~\eqref{eq: converu2}, it is sufficient to prove that 
\[
\lim_{n\goto\infty} 
\int_0^T
\iprod{|\vecu_n|^2(t)\vecu_n(t),\Pi_n\vecphi}_{\mL^2}\dt
=
\int_0^T
\iprod{|\vecu|^2(t)\vecu(t),\vecphi}_{\mL^2}\dt.
\]
By using the triangle and H\"older inequalities,~\eqref{eq: un3} and Lemma~\ref{lem: appSo_sta}, we see that
\begin{align*}
\bigg|
\int_0^T
&\iprod{|\vecu_n|^2(t)\vecu_n(t),\Pi_n\vecphi}_{\mL^2}\dt
-
\int_0^T
\iprod{|\vecu|^2(t)\vecu(t),\vecphi}_{\mL^2}\dt
\bigg|\\
&\leq
\bigg|\int_0^T\iprod{|\vecu_n|^2(t)\vecu_n(t),\Pi_n\vecphi-\vecphi}_{\mL^2}\dt\bigg|
+
\bigg|\int_0^T
\iprod{|\vecu_n|^2(t)(\vecu_n(t)-\vecu(t)),\vecphi}_{\mL^2}\dt\bigg|\\
&\quad+
\bigg|\int_0^T
\iprod{(|\vecu_n|^2(t)-|\vecu|^2(t))\vecu(t),\vecphi}_{\mL^2}\dt\bigg|\\
&\leq
\|\Pi_n\vecphi-\vecphi\|_{\mL^2}\int_0^T\|\vecu_n^3(t)\|_{\mL^2}\dt\\
&\quad+
\||\vecu_n|^2\|_{L^2(0,T;\mL^2)}\|\vecu_n-\vecu\|_{L^4(0,T;\mL^4)}\|\vecphi\|_{L^4(0,T;\mL^4)}\\
&\quad
+
\|\vecu_n-\vecu\|_{L^4(0,T;\mL^4)}
\|\vecu_n+\vecu\|_{L^4(0,T;\mL^4)}
\|\vecu\|_{L^4(0,T;\mL^4)}
\|\vecphi\|_{L^4(0,T;\mL^4)}\\
&\leq
C\|\Pi_n\vecphi-\vecphi\|_{\mL^2}
+
C\|\vecu_n-\vecu\|_{L^4(0,T;\mL^4)}.
\end{align*}
Hence,~\eqref{eq: conver4} follows by passing to the limit as $n$ tends to infinity of the above inequality and 
using~\eqref{eq: converu}.
\end{proof}
We wish to use the notations $\Delta\vecu$ and 
$\vecu\times\Delta\vecu$ in the equation satisfied by $\vecu$. 
These notations are defined as follow.

From Lemma~\ref{lem: appSo_sta}, we have $\Delta\vecu_n$ is 
uniformly bounded in $L^2(0,T;\mL^2)$. Thus, there exist
a subsequence of $\{\Delta\vecu_n\}$ (still denoted by $\{\Delta\vecu_n\}$) 
and $\vecY\in L^2(0,T;\mL^2)$ such that 
\[
\Delta\vecu_n\goto\vecY\text { weakly in } L^2(0,T;\mL^2).
\]
Together with~\eqref{eq: converu2} we obtain
\[
\int_0^T
\inprod{\vecY(t)}{\vecphi}_{\mL^2}\dt
=
-
\int_0^T
\inprod{\nabla\vecu(t)}{\nabla\vecphi}_{\mL^2}\dt,
\]
for 
$\vecphi\in\mW^{1,4}(D)\cap X^{\beta}$. 
\begin{notation}\label{no: nota1}
By denoting $\Delta\vecu:=\vecY$, we have 
$\Delta\vecu\in L^2(0,T;\mL^2)$.
\end{notation}

From~\eqref{eq: converZ} and \eqref{eq: conver5}, for 
$\vecphi\in\mW^{1,4}(D)\cap X^{\beta}$ we have
\[
\int_0^T
\inprod{\vecZ(t)}{\vecphi}_{\mL^2}\dt
=
-
\int_0^T
\iprod{\vecu(t)\times\nabla\vecu(t),\nabla\vecphi}_{\mL^2}\dt.
\]
\begin{notation}\label{no: nota2}
By denoting $\vecu\times\Delta\vecu:=\vecZ$, we have 
$\vecu\times\Delta\vecu\in L^2(0,T;\mL^{3/2})$.
\end{notation}


We now ready to prove the main theorem.
\begin{proof}{Proof of theorem~\ref{theo: main}}

For any test function $\vecphi\in\mW^{1,4}(D)\cap X^{\beta}$, from~\eqref{eq: GaLLB2} 
and integrating by parts, we have
\begin{align*}
\iprod{\vecu_n(t),\vecphi}_{\mL^2}
&=\iprod{\vecu_n(0),\vecphi}_{\mL^2}
-\kappa_1\int_0^t\iprod{\nabla\vecu_n(s),\nabla\vecphi}_{\mL^2}\ds
+\gamma\int_0^t\iprod{\Pi_n\bigl(\vecu_n(s)\times\Delta\vecu_n(s)\bigr),\vecphi}_{\mL^2}\ds\\
&\quad
-\kappa_2\int_0^t\iprod{\Pi_n\bigl((1+\mu|\vecu_n|^2(s))\vecu_n(s)\bigr),\vecphi}_{\mL^2}\ds.
\end{align*}
By passing to the limit as $n$ tends to infinity of the above equation
and using~\eqref{eq: converu}--\eqref{eq: converu2} and Lemma~\ref{lem: conver2}, 
we obtain that $\vecu$ satisfies~\eqref{eq: weakLLB}. 

Furthermore, using Notations~\ref{no: nota1}--\ref{no: nota2}, we infer that 
$\vecu$ satisfies the following equation in $X^{-\beta}$ with $\beta>\frac{4+d}{4m}$,
\begin{align}\label{eq: LLB3}
\vecu(t) = \vecu_0
+
\kappa_1
\int_0^t
\Delta\vecu(s)\ds
+
\gamma
\int_0^t
\vecu(s)\times\Delta\vecu(s)\ds
-
\kappa_2
\int_0^t
(1+|\vecu|^2(s))\vecu(s)\ds.
\end{align}

\underline{Proof of (a):} 
It is enough to prove that the terms in equation~\eqref{eq: LLB3} are in the 
space $\mL^{3/2}$. Since we wish to use the following arguments in 
the proof of (b), we will use $\int_{\tau}^t$ for $\tau\in[0,t)$ instead of just $\int_0^t$.
By using the Minkowski inequality and the continuous embedding $\mL^2\hookrightarrow\mL^{3/2}$, 
we have
\begin{align}\label{eq: a1}
\bigg\|\int_{\tau}^t
\Delta\vecu(s)\ds\bigg\|_{\mL^{3/2}}
\leq
\int_{\tau}^t
\|
\Delta\vecu(s)\|_{\mL^{3/2}}\ds
\leq
C
\int_{\tau}^t
\|
\Delta\vecu(s)\|_{\mL^2}\ds
\leq 
C(t-\tau)^{\tfrac12},
\end{align}
and
\begin{align}\label{eq: a2}
\|\int_{\tau}^t
\vecu(s)\times\Delta\vecu(s)\ds\|_{\mL^{3/2}}
&\leq
\int_{\tau}^t
\|
\vecu(s)\times\Delta\vecu(s)\|_{\mL^{3/2}}\ds\nn\\
&\leq
(t-\tau)^{\tfrac12}
\bigl(\int_{\tau}^t
\|\vecu(s)\times\Delta\vecu(s)\|^2_{\mL^{3/2}}\ds\bigr)^{\tfrac12}
\leq
C(t-\tau)^{\tfrac12}.
\end{align}
For the last term in~\eqref{eq: LLB3}, it is sufficient to 
prove that $\|\int_0^t
|\vecu|^2(s)\vecu(s)\ds\|_{\mL^{3/2}}<\infty$. Indeed, 
by using the H\"older and Minkowski inequalities we deduce
\begin{align}\label{eq: a3}
\bigg\|\int_{\tau}^t
|\vecu|^2(s)\vecu(s)\ds\bigg\|_{\mL^{3/2}}^{3/2}
&=
\int_D
\bigg|\int_{\tau}^t
|\vecu|^2(s)\vecu(s)\ds\bigg|^{3/2}\dvx\nn\\
&\leq
\int_D
\bigl(\int_{\tau}^t
|\vecu|^4(s)\ds\bigr)^{\tfrac34}\bigl(\int_{\tau}^t
|\vecu|^2(s)\ds\bigr)^{\tfrac34}\dvx\nn\\
&\leq
\bigl(\int_D
\int_{\tau}^t
|\vecu|^4(s)\ds\dvx\bigr)^{\tfrac34}
\bigl(\int_D\bigl(\int_{\tau}^t
|\vecu|^2(s)\ds\bigr)^3\dvx\bigr)^{\tfrac14}\nn\\
&\leq
(t-\tau)^{\tfrac38}
\|\vecu\|^3_{L^8(0,T;\mL^4)}
\|\vecu\|^{3/2}_{L^2(0,T;\mL^6)}
\leq 
C(t-\tau)^{\tfrac38},
\end{align}
where the last inequality follows because the fact 
that 
\[
\vecu\in L^8(0,T;\mL^4)\cap L^2(0,T;\mL^6),
\]
which is 
a consequence of 
Lemma~\ref{lem: conversub} and the embedding 
\[
L^2(0,T;\mH^1)\hookrightarrow L^2(0,T;\mL^6).
\]
By taking $\tau=0$ in~\eqref{eq: a1}--\eqref{eq: a3}, we infer that 
$\vecu$ satisfies~\eqref{eq: LLB3} in $\mL^{3/2}$.

\underline{Proof of (b):} 
From~\eqref{eq: a1}--\eqref{eq: a3}, we obtain
\[
\sup_{0\leq\tau<t\leq T} 
\frac{\|\vecu(t)-\vecu(\tau)\|_{\mL^{3/2}}}{|t-\tau|^{1/4}}
<\infty;
\]
it follows that $\vecu\in C^{\bar \alpha}([0,T];\mL^{3/2})$ 
for evey $\bar \alpha\in(0,\tfrac14]$.

\underline{Proof of (c):} 
Finally, property (c) follows from 
applying weak lower semicontinuity of norms in the first inequality of Lemma~\ref{lem: appSo_sta}, 
which complete the proof of our main theorem.
\end{proof}

\section{Appendix}
\begin{lemma}\label{lem: Ap5}
Let $X^{\beta}_n:=\{\Pi_n\vecx : \vecx\in X^{\beta}\}$ with the 
norm inherited from $X^{\beta}$. Then 
\[
\lim_{n\goto\infty} \|\Pi_n\vecx-\vecx\|_{X^{\beta}} 
= 0
\quad\text{for every $\vecx\in X^{\beta}$.}
\]
\end{lemma}
\begin{proof}
For $\vecx\in X^{\beta}$, we have 
$\Pi_n\vecx = \sum_{i=1}^n\inpro{\vecx}{\vece_i}_{\mL^2}\vece_i$, thus
$
\vecx-\Pi_n\vecx = \sum_{i=n+1}^{\infty}\inpro{\vecx}{\vece_i}_{\mL^2}\vece_i.
$
By using orthonormal property of $\{\vece_i\}$, we obtain
\[
\lim_{n\goto\infty}\|\Pi_n\vecx-\vecx\|_{X^{\beta}} 
=
\lim_{n\goto\infty}\sum_{i=n+1}^{\infty}(1+\lambda_i)^\beta\inpro{\vecx}{\vece_i}^2_{\mL^2}
=0,
\]
as $\|\vecx\|_{X^{\beta}}:= \sum_{i=1}^{\infty}(1+\lambda_i)^\beta\inpro{\vecx}{\vece_i}^2_{\mL^2}<\infty$.
\end{proof}


For the reader's convenience we will recall some embedding results that 
are crucial for the proof of convergence of the approximating sequence $\{\vecu_n\}$.
\begin{lemma}\label{lem: Ap1}\cite[Theorem 1.6.1]{Henry1981}
Suppose $\Omega\subset\R^d$ is an open set having the $C^m$ extension property, 
$1\leq p<\infty$ and $A$ is a sectorial operator in $X=\mL^p(\Omega)$ 
with $D(A) = X^1\hookrightarrow \mW^{m,p}(\Omega)$ for some $m\geq 1$. Then 
for $0\leq\beta\leq 1$,
\begin{align*}
X^{\beta}\hookrightarrow \mW^{k,q}(\Omega)
&\quad\text{when}\quad
k-\frac{d}{q}<m\beta-\frac{d}{p},\quad q\geq p,\\
\text{and}\quad
X^{\beta}\hookrightarrow \C^{\alpha}(\Omega)
&\quad\text{when}\quad
0\leq \alpha < m\beta-\frac{d}{p}.
\end{align*}
\end{lemma}
\begin{lemma}\label{lem: Ap2}\cite[Corollary 19]{Simon1990}
Suppose $s\geq r$, $p\leq q$ and $s-1/p\geq r-1/q$ 
($0<r\leq s<1$, $1\leq p\leq q\leq \infty$). Let $E$ be a 
Banach space and $I$ be an interval of $\R$. Then
\[
W^{s,p}(I;E)\hookrightarrow W^{r,q}(I;E).
\]
\end{lemma}
\begin{lemma}\label{lem: Ap3}\cite[Theorem 2.1]{Flan95}
Assume that $B_0\subset B\subset B_1$ are Banach spaces, 
$B_0$ and $B_1$ reflexive with compact embedding of $B_0$ 
in $B$. Let $p\in(1,\infty)$ and $\alpha\in(0,1)$ be given. 
Then the embedding 
\[
L^p(0,T;B_0)\cap W^{\alpha,q}(0,T;B_1)
\hookrightarrow
L^p(0,T;B)
\,\text{ is compact.}
\]
\end{lemma}
\begin{lemma}\label{lem: Ap4}\cite[Theorem 2.2]{Flan95}
Assume that $B_0\subset B$ are Banach spaces such that the 
embedding  $B_0\hookrightarrow B$ is compact. Let 
$p\in(1,\infty)$ and $0<\alpha<1$ and $\alpha p>1$. Then 
the embedding 
\[
W^{\alpha,q}(0,T;B_0)
\hookrightarrow
C([0,T];B)
\,\text{ is compact.}
\]
\end{lemma}

\end{document}

%% file: ams_notation.tex
\newcommand{\C}{{\mathbb C}}
\newcommand{\R}{{\mathbb R}}

\newcommand{\N}{{\mathbb N}}

\newcommand{\mS}{\mathbb S}

\newcommand{\mH}{\mathbb H}
\newcommand{\mL}{\mathbb L}

\newcommand{\mW}{\mathbb W}

\newcommand{\inpro}[2]{\left\langle{#1},{#2}\right\rangle}
\newcommand{\inprod}[2]{\left\langle{#1},{#2}\right\rangle}

\newcommand{\vecB}{\boldsymbol{B}}

\newcommand{\vecH}{\boldsymbol{H}}

\newcommand{\vecY}{\boldsymbol{Y}}

\newcommand{\vecZ}{\boldsymbol{Z}}

\newcommand{\veca}{\boldsymbol{a}}
\newcommand{\vecb}{\boldsymbol{b}}
\newcommand{\vecc}{\boldsymbol{c}}

\newcommand{\vece}{\boldsymbol{e}}

\newcommand{\vecm}{\boldsymbol{m}}

\newcommand{\vecu}{\boldsymbol{u}}
\newcommand{\vecv}{\boldsymbol{v}}
\newcommand{\vecw}{\boldsymbol{w}}
\newcommand{\vecx}{\boldsymbol{x}}

\newcommand{\vecphi}{\boldsymbol{\phi}}

\newcommand{\goto}{\rightarrow}

\newcommand{\p}{\partial}

\newcommand{\ds}{\, ds}
\newcommand{\dt}{\, dt}

\newcommand{\dvx}{\, d\vecx}

\newcommand{\nn}{\nonumber}

%% file: env.tex
\numberwithin{equation}{section}
\newtheorem{theorem}{Theorem}[section]
\newtheorem{lemma}[theorem]{Lemma}

\newtheorem{remark}[theorem]{Remark}
\newtheorem{definition}[theorem]{Definition}

